\documentclass[a4paper,11pt]{article}
\usepackage[utf8]{inputenc}
\usepackage[T1]{fontenc}
\usepackage{amsfonts,epsf,amsmath,amssymb,amsthm}
\usepackage{tgtermes}
\usepackage{enumerate}

\usepackage{tikz}
\usepackage{url}
\usepackage{fixltx2e}

\usepackage[hidelinks]{hyperref}
\usepackage{fge}

\usepackage[margin=1in]{geometry}
\usepackage{graphicx}
\usepackage{centernot}

\renewcommand{\theenumi}%
{\arabic{enumi}}
\renewcommand{\theenumii}%
{\theenumi.\arabic{enumii}}
\renewcommand{\theenumiii}%
{\theenumii.\arabic{enumiii}}
\renewcommand{\labelenumii}%
{\theenumii.}

\newtheorem{theorem}{Theorem}[section]
\newtheorem{lemma}[theorem]{Lemma}
\newtheorem{proposition}[theorem]{Proposition}
\newtheorem{corollary}[theorem]{Corollary}

\newtheorem{open}[theorem]{Open problem}

\renewenvironment{proof}[1][Proof.]{\begin{trivlist}
\item[\hskip \labelsep {\bfseries #1}]}{\end{trivlist}}

\newcommand{\var}[1]{\mathbf{#1}}

\def\qed{\ifhmode\unskip\nobreak\fi\hfill
  \ifmmode\square\else$\square$\fi}

\newcommand{\RR}{\mathbb{R}}

\newcommand{\NN}{\mathbb{N}}
\newcommand{\ZZ}{\mathbb{Z}}
\newcommand{\EE}[1]{\mathbb{E}(#1)}

\newcommand{\oo}{\operatorname{o}}
\newcommand{\OO}{\operatorname{O}}

\newcommand{\TT}{\operatorname{\Theta}}

\newcommand{\ie}{i.e. }
\newcommand{\eg}{e.g. }

\title{The limit of binomial means of a sequence\footnote{This work is partially funded by the Slovenian Research Agency.}}
\date{}
\author{David Gajser\\ IMFM, Jadranska 19, 1000 Ljubljana, Slovenija \\\href{mailto:david.gajser@fmf.uni-lj.si}{david.gajser@fmf.uni-lj.si}}

\begin{document}
\maketitle
\thispagestyle{empty}
\paragraph{Abstract.}  For a sequence $\{a_n\}_{n\geq 0}$ of real numbers and for a parameter $0<p<1$, we define the sequence of its arithmetic means $\{a^*_n\}_{n\geq 0}$ and the sequence of its $p$-binomial means $\{a^p_n\}_{n\geq 0}$ as
\begin{align*}
a^*_n=\frac{1}{n+1}\sum_{i=0}^n a_i & & \textrm{and} && a^p_n=\sum_{i=0}^n\binom{n}{i}p^i(1-p)^{n-i} a_i.\\
\end{align*}
We compare the convergence of sequences $\{a_n\}_{n\geq 0}$, $\{a_n^*\}_{n\geq 0}$ and $\{a_n^p\}_{n\geq 0}$ for various $0<p<1$, \ie we analyze when the convergence of one sequence implies the convergence of the other.

While the sequence $\{a^*_n\}_{n\geq 0}$, known also as the sequence of Ces\`{a}ro means of a sequence, is well studied in the literature, the results about $\{a^p_n\}_{n\geq 0}$ are hard to find. Our main result shows that, if $\{a_n\}_{n\geq 0}$ is a sequence of non-negative
real numbers such that $\{a^p_n\}_{n\geq 0}$ converges to $a\in\RR\cup\{\infty\}$ for some $0<p<1$, then $\{a^*_n\}_{n\geq 0}$ also converges to $a$. We give an application of this result on finite Markov chains.

\vspace{-0.5cm}
\paragraph{Math. Subj. Class.} 00A05

\newpage
\pagenumbering{arabic}
\section{Introduction}

For a sequence $\{a_n\}_{n\geq 0}$ of real numbers and for a parameter $0<p<1$, define the sequence of its arithmetic means $\{a^*_n\}_{n\geq 0}$ and the sequence of its $p$-binomial means $\{a^p_n\}_{n\geq 0}$ as
\begin{align*}
a^*_n=\frac{1}{n+1}\sum_{i=0}^n a_i & & \textrm{and} && a^p_n=\sum_{i=0}^n \binom{n}{i}p^i(1-p)^{n-i}a_i.\\
\end{align*}
We see that $a^*_n$ is an uniformly weighted average of numbers $a_0, a_1\ldots a_n$ and $a^p_n$ is a binomially weighted average of numbers $a_0, a_1\ldots a_n$.

In this article, we will analyse the relationship between the convergence of sequences $\{a_n\}_{n\geq 0}$, $\{a^p_n\}_{n\geq 0}$ and $\{a^*_n\}_{n\geq 0}$. Our results are presented in the following table.
%
\begin{table}[htb]
\centering
{\Large
\begin{tabular}{c|c|c|c|c|}
                                     &$\{a_n\}_{n\geq 0}$&$\{a^p_n\}_{n\geq 0}$&$\{a^q_n\}_{n\geq 0}$&$\{a_n^*\}_{n\geq 0}$\\
\hline
$\{a_n\}_{n\geq 0}$ & $\implies$& $\implies$ & $\implies$& $\implies$\\
$\{a^p_n\}_{n\geq 0}$ & $\centernot\implies$&$\implies$ & $\overset{\textrm{? }a_n\geq 0}{\implies}$& $\overset{a_n\geq 0}{\implies}$\\
$\{a^q_n\}_{n\geq 0}$ &$\centernot\implies$ & $\implies$&$\implies$ & $\overset{a_n\geq 0}{\implies}$\\
$\{a^*_n\}_{n\geq 0}$ &$\centernot\implies$ & $\centernot\implies$&$\centernot\implies$ & $\implies$
\end{tabular}
\caption{The table shows whether the convergence of a sequence on the left implies the convergence of a sequence above, for $0<p<q<1$. The symbol $\implies$ means that the implication holds and the symbol $\centernot\implies$ means that there is a counterexample with $a_n\in\{0,1\}$, for all $n\in\NN$. If there is a condition above $\implies$, then the implication does not hold in general, but it holds if the condition is true. If there is  ? before the condition, we do not know whether the condition is the right one (open problem), but the implication does not hold in general.}
\label{tabela}
}
\end{table}


 The sequence $\{a^*_n\}_{n\geq 0}$ is also known as the sequence of Ces\`{a}ro means and is well studied in the literature~\cite{Boos,Cezaro}. On the other hand, information about the convergence of $p$-binomial means is hard to find. Also the notion of $p$-binomial means is coined especially for the purpose of this article.
 However, there are a few definitions that are close to ours~\cite{Boos,Cezaro, blizu1}. First, we have to mention the Hausdorff means~\cite{Boos,Cezaro}: the $p$-binomial means as well as the arithmetic mean are its special cases. Unfortunately, the Hausdorff means are a bit too general for our purposes in the sense that the known results that are useful for this paper, can be quite easily proven in our special cases.

One of the most similar notions to the $k$-binomial mean is the one of $k$-binomial transform~\cite{blizu1}:
$$\tilde{a}^k_n =\sum_{i=0}^n\binom{n}{i}k^{n} a_i,$$
which coincides with $\{a_n^{p}\}_{n\geq 0}$ for $k=p=0.5$, but is different for other $p$ and $k$. Another similar definition is given with Euler means~\cite[pages 70, 71]{Cezaro}:
$$\overline{a}_n =\frac{1}{2^{n+1}}\sum_{i=0}^n\binom{n+1}{i+1} a_i.$$

Some results, like the first row and the first column of Table~\ref{tabela}, are not hard to prove (Section~\ref{prviRez}) and the diagonal is trivial. Other results (Sections~\ref{compareB} and~\ref{compare}) require more careful ideas. This is true especially for the main result of this paper, Theorem~\ref{mainThm}, which proves, using the notation  from Table~\ref{tabela}, that $$\{a^p_n\}_{n\geq 0}\overset{a_n\geq 0}{\implies}\{a_n^*\}_{n\geq 0}.$$
In Section~\ref{applications} we give an application of this theorem on finite Markov chains.

\section{Preliminaries}
	\label{prelim}

Let $\NN$, $\RR^+$ and $\RR^+_0$ be sets of non-negative integers, positive real numbers and non-negative real numbers, respectively. For $a\in\RR$, let $\lfloor a\rfloor$ be the greatest integer not greater than $a$ and let $\lceil a\rceil$ be the smallest integer not smaller than $a$. We will allow a limit of a sequence to be infinite and we will write $a<\infty$ (which means exactly $a\in\RR$) to emphasize that $a$ is finite.

For functions $f,g:\NN\rightarrow\RR^+_0$ we say that 
\begin{itemize}
\item $f(n)=\OO(g(n))$ if there is some $C>0$ such that $f(n)\leq Cg(n)$ for all sufficiently large $n$,
\item $f(n)=\TT(g(n))$ if there are some $C_1,C_2>0$ such that $C_1g(n)\leq f(n)\leq C_2g(n)$ for all sufficiently large $n$,
\item $f(n)=\oo(g(n))$ if $g(n)$ is non-zero for all large enough $n$ and $\displaystyle\lim_{n\rightarrow\infty}\frac{f(n)}{g(n)}=0$.
\end{itemize}



The following lemma will be useful later.
\begin{lemma}
	\label{limita}
Let $u:\NN\rightarrow\RR\backslash\{0\}$ and $k:\NN\rightarrow\RR$ be functions such that $\displaystyle\lim_{n\rightarrow\infty}u(n)k(n)=\lim_{n\rightarrow\infty}u(n)=0$. Then
$$\lim_{n\rightarrow\infty}\frac{\big(1+u(n)\big)^{k(n)/u(n)}}{e^{k(n)}}=1.$$
\end{lemma}
\begin{proof}
Because $e^x=\sum\frac{x^i}{i!}$ and  $e^x\geq1+x$, there is an analytic function $g:\RR\rightarrow\RR^+$ such that $e^x=1+x+g(x)x^2$ and $g(0)=\frac{1}{2}$. 
Hence, if we omit writing the argument of functions $u$ and $k$,
\begin{align*}
\lim_{n\rightarrow\infty}\frac{(1+u)^{k/u}}{e^{k}}
=\lim_{n\rightarrow\infty}\left(\frac{e^u-g(u)u^2}{e^u}\right)^{k/u}
= \lim_{n\rightarrow\infty}\left(\left(1-\frac{g(u)u^2}{e^{u}}\right)^{\frac{e^u}{g(u)u^2}}\right)^{\frac{ukg(u)}{e^u}}.
\end{align*}
Because $\displaystyle\lim_{n\rightarrow \infty}\frac{g(u)u^2}{e^u}=0$ and because $\displaystyle\lim_{x\rightarrow 0}(1-x)^{1/x}=e^{-1}$, we have
$$\lim_{n\rightarrow\infty}\left(1-\frac{g(u)u^2}{e^u}\right)^{\frac{e^u}{g(u)u^2}}=e^{-1}.$$
From 
$$\lim_{n\rightarrow\infty}\frac{ukg(u)}{e^u}=0,$$
the result follows.\qed
\end{proof}

\subsection*{Some properties of probability mass function of binomial distribution}
	\label{binomials}

Let $\var{X}$ be a random variable having a binomial distribution with parameters $p\in(0,1)$ and $n\in\NN$. For $i\in\ZZ$, we have by definition
\begin{align*}
\Pr[\var{X}=i]=B_n^i(p)=
\left\{
	\begin{array}{ll}
		\binom{n}{i}p^i(1-p)^{n-i}& \mbox{if }0\leq i\leq n\\
		0& \mbox{else.}
	\end{array}
\right.
\end{align*}
In this subsection, we state and mathematically ground some properties that can be seen from a graph of binomial distribution (see Fig.~\ref{slikaBinom}). The results will be nice, some of them folklore, but the proofs will be technical.

\begin{figure}[hbt]
\centering\includegraphics[width=8cm]{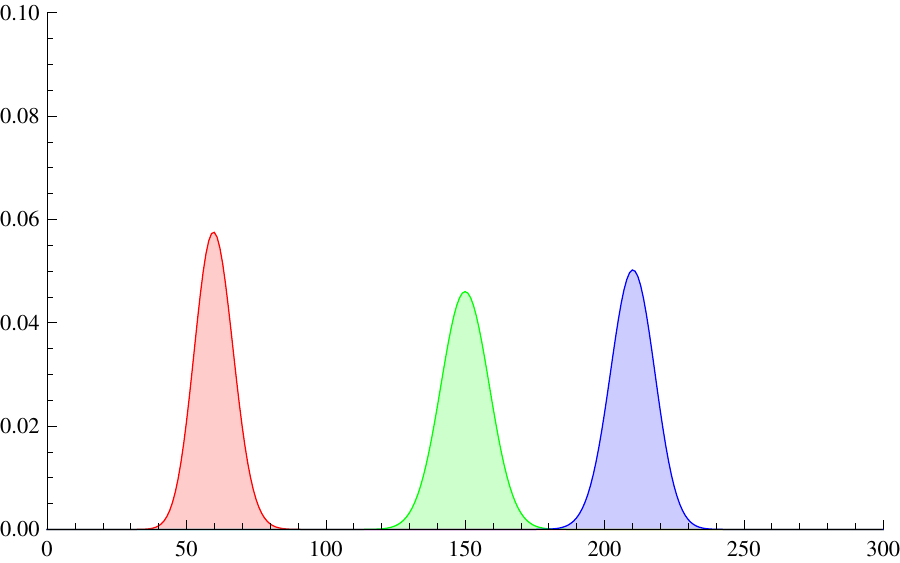}
\caption{Binomial distribution with $n=300$ and $p=0.2$ (red), $p=0.5$ (green), $p=0.7$ (blue). The graphs show $B_n^i(p)$ with respect to $i$.}
	\label{slikaBinom}
\end{figure}

 It is well known (see some basic probability book) that the expected value of $\var{X}$ is $\EE{\var{X}}=pn$. First, we will prove that also the ``peak'' of the probability mass function is roughly at $pn$.

\begin{lemma}
For $p\in(0,1)$, $n\in\NN$ and for $0\leq i\leq n$,
$$B_n^{i}(p)\geq B_n^{i-1}(p)\iff i\leq (n+1)p.$$
\end{lemma}
\begin{proof}
The expression
\begin{align*}
\frac{B_n^{i}(p)}{B_n^{i-1}(p)}&=\frac{(n-i+1)p}{i(1-p)}
\end{align*}
is at least 1 iff $i\leq p(n+1)$.\qed
\end{proof}

Next, we state a Chernoff bound proven in~\cite[inequalities~(6) and~(7)]{Hagerup}, which explains why the  probability mass function for binomial distribution ``disappears'' (see Fig.~\ref{slikaBinom}), when $i$ is far enough from $pn$.

\begin{theorem}
	\label{chernoff}
Let $\var{X}$ be a binomially distributed random variable with parameters $p\in(0,1)$ and $n\in\NN$. Then for each $\delta\in(0,1)$,
$$\Pr\big[|\var{X}-np|\geq np\delta\big]\leq 2e^{-\delta^2np/3}.$$
\end{theorem}
We will only use the following corollary of the theorem. Its proof is left to the reader.
\begin{corollary}
	\label{chernoffc}
For $p\in(0,1)$, let $\alpha:\NN\rightarrow\RR^+$ be some function such that $\alpha(n)< p\sqrt{n}$ for all $n$. Then, for all $n\in\NN$, it holds
$$\sum_{i:\,|i-np|\geq\sqrt{n}\alpha(n)}B_n^i(p)\leq 2e^{-\alpha^2(n)/(3p)}.$$
\end{corollary}

This corollary also tells us that for large $n$, roughly everything is gathered in an $\OO(\sqrt{n})$ neighborhood of $np$. What is more, the next lemma implies that in $\oo(\sqrt{n})$ neighborhood of $np$, $B^i_n(p)$ does not change a lot.

\begin{lemma}
	\label{okolicentra}
Let $p\in (0,1)$ be a parameter and let $\beta(n):\NN\rightarrow\RR$ be a function such that $|\beta(n)|=\OO(\sqrt{n})$ and $\displaystyle\lim_{n\rightarrow\infty}|\beta(n)|=\infty$. Then, for all large enough $n$, it holds
$$\frac{B^{\lfloor np\rfloor}_n(p)}{B^{\lfloor np\rfloor-\lfloor\beta(n)\rfloor}_n(p)}\leq e^{\frac{1}{p(1-p)}\cdot\frac{\lfloor\beta(n)\rfloor^2}{n}}.$$
\end{lemma}
\begin{proof}
For all large enough $n$ for which $\beta(n)\geq 0$, we have
\begin{align*}
\frac{B^{\lfloor np\rfloor}_n(p)}{B^{\lfloor np\rfloor-\lfloor\beta(n)\rfloor}_n(p)}
&=\frac{\binom{n}{\lfloor np\rfloor}p^{\lfloor\beta(n)\rfloor}}{\binom{n}{\lfloor np\rfloor-\lfloor\beta(n)\rfloor}(1-p)^{\lfloor\beta(n)\rfloor}}\\
&=\prod_{i=0}^{\lfloor\beta(n)\rfloor-1}\frac{(n-\lfloor np\rfloor+\lfloor\beta(n)\rfloor-i)p}{(\lfloor np\rfloor-i)(1-p)}\\
&\leq \prod_{i=0}^{\lfloor\beta(n)\rfloor-1} \left(1+\frac{1}{p(1-p)}\cdot\frac{\lfloor\beta(n)\rfloor}{n}\right).
\end{align*}
In the last inequality we used the fact that
\begin{equation*}
	\label{pomoc}
\frac{(n-\lfloor np\rfloor+\lfloor\beta(n)\rfloor-i)p}{(\lfloor np\rfloor-i)(1-p)}\leq \left(1+\frac{1}{p(1-p)}\cdot\frac{\lfloor\beta(n)\rfloor}{n}\right)
\end{equation*}
 holds for large enough $n$, which is true because it is equivalent to
$$(np-\lfloor np \rfloor)+(\lfloor\beta(n)\rfloor-i)p+i(1-p)+\frac{i\lfloor\beta(n)\rfloor}{pn}\leq \frac{\lfloor np \rfloor}{np}\lfloor\beta(n)\rfloor,$$
where
\begin{itemize}
\item $np-\lfloor np \rfloor\leq 1$,
\item $(\lfloor\beta(n)\rfloor-i)p+i(1-p)\leq \lfloor\beta(n)\rfloor\cdot\max\{p,1-p\}$, since $i<\lfloor\beta(n)\rfloor$ and
\item $\frac{i\lfloor\beta(n)\rfloor}{pn}=\OO(1)$, since $\beta(n)=\OO(\sqrt{n})$.
\end{itemize}

Using the fact that $(1+x)\leq e^x$ for all $x\in\RR$, we see that
\begin{align*}
\frac{B^{\lfloor np\rfloor}_n(p)}{B^{\lfloor np\rfloor-\lfloor\beta(n)\rfloor}_n(p)}
&\leq\prod_{i=0}^{\lfloor\beta(n)\rfloor-1} \left(1+\frac{1}{p(1-p)}\cdot\frac{\lfloor\beta(n)\rfloor}{n}\right)\\
&\leq\prod_{i=0}^{\lfloor\beta(n)\rfloor-1} e^{\frac{1}{p(1-p)}\cdot\frac{\lfloor\beta(n)\rfloor}{n}}\\
&= e^{\frac{1}{p(1-p)}\cdot\frac{\lfloor\beta(n)\rfloor^2}{n}}.
\end{align*}

For all large enough $n$ for which $\beta(n)<0$, we write $b(n)=|\lfloor\beta(n)\rfloor|$ and we have
\begin{align*}
\frac{B^{\lfloor np\rfloor}_n(p)}{B^{\lfloor np\rfloor-\lfloor\beta(n)\rfloor}_n(p)}
&=\frac{\binom{n}{\lfloor np\rfloor}(1-p)^{b(n)}}{\binom{n}{\lfloor np\rfloor+b(n)}p^{b(n)}}\\
&=\prod_{i=0}^{b(n)-1}\frac{(\lfloor np\rfloor +b(n)-i)(1-p)}{(n-\lfloor np\rfloor-i)p}\\
&\leq \prod_{i=0}^{b(n)-1}\frac{(np +b(n)-i)(1-p)}{(n(1-p)-i)p}\\
&\leq \prod_{i=0}^{b(n)-1}\frac{(n-\lfloor n(1-p)\rfloor +b(n)-i)(1-p)}{(\lfloor n(1-p)\rfloor-i)p},
\end{align*}
which is the same as by $\beta(n)\geq 0$, only that $p$ and $(1-p)$ are interchanged. \qed
\end{proof}
Now we know that the values of $B^i_n(p)$ around the peaks on Fig.~\ref{slikaBinom} are close  to the value of the peak. The next lemma will tell us that the peak of $B^i_n(p)$ is asymptotically $\frac{1}{\sqrt{2\pi p(1-p)n}}$.

\begin{lemma}
	\label{osrednji}
For $0<p<1$, it holds
\begin{equation*}
\lim_{n\rightarrow\infty}\sqrt{2\pi p(1-p)n}B^{\lfloor np\rfloor}_n(p)=1.
\end{equation*}
\end{lemma}
\begin{proof}
Using Stirling's approximation
$$\lim_{n\rightarrow\infty}\frac{n!}{\sqrt{2\pi n}\left(\frac{n}{e}\right)^n}=1,$$
we see that
\begin{align*}
\lim_{n\rightarrow\infty}\sqrt{2\pi p(1-p)n}&B^{\lfloor np\rfloor}_n(p)\\
&=\lim_{n\rightarrow\infty}\frac{\sqrt{2\pi p(1-p)n} \cdot \sqrt{2\pi n}\left(\frac{n}{e}\right)^n p^{\lfloor np\rfloor}(1-p)^{n-\lfloor np\rfloor}}{\sqrt{2\pi\lfloor np\rfloor}\left(\frac{\lfloor np\rfloor}{e}\right)^{\lfloor np\rfloor}\cdot \sqrt{2\pi (n-\lfloor np\rfloor)}\left(\frac{n-\lfloor np\rfloor}{e}\right)^{n-\lfloor np\rfloor}}\\
&=\lim_{n\rightarrow\infty}\frac{n^n p^{\lfloor np\rfloor}(1-p)^{n-\lfloor np\rfloor}}{\lfloor np\rfloor^{\lfloor np\rfloor}\cdot (n-\lfloor np\rfloor)^{n-\lfloor np\rfloor}}\\
&=\lim_{n\rightarrow\infty}\left(\frac{np}{\lfloor np\rfloor}\right)^{\lfloor np\rfloor}\cdot\left(\frac{n-np}{n-\lfloor np\rfloor}\right)^{n-\lfloor np\rfloor}\\
&=\lim_{n\rightarrow\infty}\left(1+\frac{np-\lfloor np\rfloor}{\lfloor np\rfloor}\right)^{\lfloor np\rfloor}\cdot\left(1-\frac{np-\lfloor np\rfloor}{n-\lfloor np\rfloor}\right)^{n-\lfloor np\rfloor}\\
&=\lim_{n\rightarrow\infty}e^{np-\lfloor np\rfloor}\cdot e^{-(np-\lfloor np\rfloor)}=1,
\end{align*}
where the last line follows by Lemma~\ref{limita} (we can restrict ourselves only to those $n$ for which $np\not\in\NN$).\qed
\end{proof}

\section{Comparing the limit of $\{a_n\}_{n\geq 0}$ with the limit of $\{a^p_n\}_{n\geq 0}$  and $\{a^*_n\}_{n\geq 0}$} 
	\label{prviRez}
In this section we present results about the relationship between the convergence of  $\{a_n\}_{n\geq 0}$ with the convergence of $\{a^p_n\}_{n\geq 0}$  and $\{a^*_n\}_{n\geq 0}$. It is well known~\cite{Cezaro} that if $\{a_n\}_{n\geq 0}$  converges to $a\in\RR\cup\{\infty\}$, then so does $\{a^*_n\}_{n\geq 0}$. The next theorem tells us that in this case, $\{a^p_n\}_{n\geq 0}$ also converges to the same limit.

\begin{theorem}
	\label{enost}
If $\{a_n\}_{n\geq 0}$  converges to $a\in\RR\cup\{\infty\}$, then $\{a^*_n\}_{n\geq 0}$ and $\{a^p_n\}_{n\geq 0}$ converge to $a$ for all $0<p<1$.
\end{theorem}
\begin{proof}
The case $a=\infty$ is left for the reader, so suppose $a<\infty$. Take any $\epsilon>0$ and such $N$ that $|a_n-a|<\frac{\epsilon}{2}$ for all $n\geq N$. Then, for $n\geq N$,
\begin{align*}
 |a^*_n-a|&=\frac{1}{n+1}\left|\sum_{i=0}^n(a_i-a)\right|\\
&\leq \frac{1}{n+1}\sum_{i=0}^n|a_i-a|\\
&\leq \frac{1}{n+1}\sum_{i=0}^N|a_i-a|+\frac{1}{n+1}\cdot\frac{\epsilon }{2}(n-N).
\end{align*}
The last line converges to $\frac{\epsilon}{2}$ when $n$ goes to infinity, which implies that  $\{a^*_n\}_{n\geq 0}$ converges to $a$.

Similarly, for $n\geq N$,
\begin{align*}
 |a^p_n-a|&=\left|\sum_{i=0}^n\binom{n}{i}p^i(1-p)^{n-i}(a_i-a)\right|\\
&\leq \sum_{i=0}^N\binom{n}{i}p^i(1-p)^{n-i}|a_i-a|+\frac{\epsilon }{2}\sum_{i=N+1}^n\binom{n}{i}p^i(1-p)^{n-i}\\
&\leq  \sum_{i=0}^N\binom{n}{i}p^i(1-p)^{n-i}|a_i-a|+\frac{\epsilon }{2}.
\end{align*}
The last line converges to $\frac{\epsilon}{2}$ because $\binom{n}{i}$ grows as a polynomial in $n$ for each fixed value $i\leq N$ and $p^i(1-p)^{n-i}$ decreases exponentially. This implies that  $\{a^p_n\}_{n\geq 0}$ also converges to $a$.\qed
\end{proof}

One does not need to go searching for strange examples to see that convergence of $\{a^*_n\}_{n\geq 0}$ or $\{a^p_n\}_{n\geq 0}$ does not imply the convergence of $\{a_n\}_{n\geq 0}$. We state this as a proposition.

\begin{proposition}
	\label{prviNeg}
There exists a sequence $\{a_n\}_{n\geq 0}$ of zeros and ones that does not converge, whereas $\{a^*_n\}_{n\geq 0}$ and $\{a^p_n\}_{n\geq 0}$ converge for all $0<p<1$.
\end{proposition}
\begin{proof}
Define
\begin{align*}
a_n =
\left\{
	\begin{array}{ll}
		0& \mbox{if } n\mbox{ is odd,} \\
		1& \mbox{if } n\mbox{ is even.}
	\end{array}
\right.
\end{align*}
Then  $\{a_n\}_{n\geq 0}$ does not converge and $\{a^*_n\}_{n\geq 0}$ converges to $\frac{1}{2}$, as the reader can verify. Next, we will prove that $\{a^p_n\}_{n\geq 0}$ converges to $\frac{1}{2}$. First, we see that, for $0<p<1$, the value of $(1-2p)$ is strictly between $-1$ and $1$, thus $(1-2p)^n$ converges to 0 when $n$ goes to infinity. Hence,
\begin{align*}
\sum_{i\textrm{ is even}}^n \binom{n}{i}p^i(1-p)^{n-i}-\sum_{i\textrm{ is odd}}^n \binom{n}{i}p^i(1-p)^{n-i}=(-p+(1-p))^n
\end{align*}
 converges to 0. Because
\begin{align*}
\sum_{i\textrm{ is even}}^n \binom{n}{i}p^i(1-p)^{n-i}+\sum_{i\textrm{ is odd}}^n \binom{n}{i}p^i(1-p)^{n-i}=1,
\end{align*}
we have that $\{a^p_n\}_{n\geq 0}$ converges to $\frac{1}{2}$.\qed
\end{proof}

\section{Comparing the limits of binomial means}
	\label{compareB}

In this section we compare the limits of sequences $\{a_n^p\}_{n\geq 0}$ for different parameters $p\in(0,1)$. We will see that if $0<p\leq q<1$, then the convergence of $\{a_n^q\}_{n\geq 0}$ implies the convergence of $\{a_n^p\}_{n\geq 0}$ to the same limit, while the convergence of $\{a_n^p\}_{n\geq 0}$ does not imply the convergence of $\{a_n^q\}_{n\geq 0}$ in general. We leave as an open problem whether for $a_n\geq 0$ it does.

First, let us prove the main lemma in this section, which tells us that the sequence of $q$-binomial means of the sequence of $p$-binomial means of some sequence is the sequence of $(pq)$-binomial means of the starting sequence.
\begin{lemma}
	\label{navzdol}
For $0<p,q<1$ and for a sequence $\{a_n\}_{n\geq 0}$, let $\{b_n\}_{n\geq 0}$ be a sequence of $p$-binomial means of $\{a_n\}_{n\geq 0}$, \ie $b_n=a_n^p$ for all $n$. Then $b_n^q=a_n^{pq}$ for all $n$.
\end{lemma}
\begin{proof}
We can change the order of summation, consider $\binom{j}{i}\binom{n}{j}=\binom{n}{i}\binom{n-i}{j-i}$ for $i\leq j$ and put $k=j-i$:
\begin{align*}
b_n^q
&=\sum_{j=0}^n a_j^p\binom{n}{j}q^j(1-q)^{n-j}\\
&=\sum_{j=0}^n\sum_{i=0}^j a_i\binom{j}{i}\binom{n}{j}p^i(1-p)^{j-i}q^j(1-q)^{n-j}\\
&=\sum_{i=0}^n a_i  \binom{n}{i}p^iq^i \sum_{j=i}^n \binom{n-i}{j-i} (1-p)^{j-i}q^{j-i}(1-q)^{n-j}\\
&=\sum_{i=0}^n a_i  \binom{n}{i}p^iq^i \sum_{k=0}^{n-i} \binom{n-i}{k} ((1-p)q)^{k}(1-q)^{n-i-k}\\
&=\sum_{i=0}^n a_i  \binom{n}{i}p^iq^i ((1-p)q+(1-q))^{n-i}\\
&=\sum_{i=0}^n a_i  \binom{n}{i}(pq)^i (1-pq)^{n-i}.
\end{align*}
The last line equals $a_n^{pq}$.\qed
\end{proof}
The next theorem will now be trivial to prove.
\begin{theorem}
	\label{pq}
For $0<p<q<1$ and for a sequence $\{a_n\}_{n\geq 0}$, if $\{a_n^q\}_{n\geq 0}$ converges to $a\in\RR\cup\{\infty\}$, then $\{a_n^p\}_{n\geq 0}$ also converges to $a$.
\end{theorem}
\begin{proof}
From Lemma~\ref{navzdol} we know that $\{a_n^p\}_{n\geq 0}$ is the sequence of $\frac{p}{q}$-binomial means of the sequence $\{a_n^q\}_{n\geq 0}$. By Theorem~\ref{enost}, it converges to $a$.\qed
\end{proof}
The next proposition tells us, that the condition $0<p<q<1$ in the above theorem cannot be left out in general.
\begin{proposition}
	\label{counterpq}
For $0<p<q<1$, there exists a sequence $\{a_n\}_{n\geq 0}$, such that $\{a_n^p\}_{n\geq 0}$ converges to 0, but $\{a_n^q\}_{n\geq 0}$ does not converge.
\end{proposition}
\begin{proof}
Define $\{a_n\}_{n\geq 0}$ as $a_n=a^n$ for some parameter $a\in\RR$. If $a>-1$, $\{a_n\}_{n\geq 0}$ converges (possibly to $\infty$), so let us examine the case when $a\leq -1$. In this case we have $$a_n^p=\sum_{i=0}^n\binom{n}{i} a^ip^i(1-p)^{n-i}=(ap+(1-p))^n=(p(a-1)+1)^n,$$
which converges iff $p<\frac{2}{1-a}$. So we can choose such an $a$ that $p<\frac{2}{1-a}<q$, \ie $1-\frac{2}{p}<a<1-\frac{2}{q}$. It follows that $\{a_n^p\}_{n\geq 0}$ converges to 0, but $\{a_n^q\}_{n\geq 0}$ does not converge.\qed
\end{proof}
The sequence $\{a_n\}_{n\geq 0}$ in the above proof is growing very rapidly by absolute value and the sign of its elements alternates. We think that this is not a coincidence and we state the following open problem.
\begin{open}
	\label{domneva}
Let $\{a_n\}_{n\geq 0}$ be a sequence of non-negative real numbers. Is it true that, for all $0<p,q<1$, the sequence $\{a^p_n\}_{n\geq 0}$ converges to $a$ iff $\{a^q_n\}_{n\geq 0}$ converges to $a$? If the answer is no, is there a counterexample where $a_n\in\{0,1\}$?
\end{open}
Note that the condition $a_n\geq 0$ is also required for the main result of the paper, Theorem~\ref{mainThm}. If the answer on~\ref{domneva} was \emph{yes}, then we would only have to prove Theorem~\ref{mainThm} in a special case, \eg for $p=\frac{1}{2}$. The (possibly negative) answer would also make this paper more complete (see Table~\ref{tabela}). For the rest of this section we will try to give some insight into this problem and we will present some reasons for why we think it is hard.

Suppose we have $0<p<q<1$ and a sequence $\{a_n\}_{n\geq 0}$ of non-negative real numbers such that $\{a^p_n\}_{n\geq 0}$ converges to $a\in\RR$ (the case when $\{a^q_n\}_{n\geq 0}$ converges is covered by Theorem~\ref{pq}). 
The next lemma implies that $\{a_n\}_{n\geq 0}$ has a relatively low upper bound on how fast its elements can increase, ruling out too large local extremes.
\begin{lemma}
	\label{koren}
Let $\{a_n\}_{n\geq 0}$ be a sequence of non-negative real numbers and let  $0<p<1$. If $\{a^p_n\}_{n\geq 0}$ converges to $a<\infty$, then $a_n=\OO(\sqrt{n})$.
\end{lemma}
\begin{proof}
We know that $a_n^p\geq a_{\lfloor np\rfloor} B_n^{\lfloor np\rfloor}(p)$, where $B_n^{\lfloor np\rfloor}(p)\approx \frac{1}{\sqrt{2\pi p(1-p)n}}$ by Lemma~\ref{osrednji} and $a_n^p\approx a$ for large $n$. Hence, $a_{\lfloor np\rfloor}=\OO(\sqrt{n})$.\qed
\end{proof}

To see whether $\{a^q_n\}_{n\geq 0}$ converges, it makes sense to compare $a_{\lfloor n/p\rfloor}^p$ with $a_{\lfloor n/q\rfloor}^q$, since the peaks of the ``weights'' $B_{\lfloor n/p\rfloor}^i(p)$ and $B_{\lfloor n/q\rfloor}^i(q)$ (roughly) coincide by $n$ (see Fig~\ref{slikaP}). 
Now the troublesome thing is that, for large $n$, the peaks are not of the same height, but rather they differ for a factor $$\sqrt{\frac{1-q}{1-p}}$$ by Lemma~\ref{osrednji}. Because the weights $B_{\lfloor n/p\rfloor}^i(p)$ and $B_{\lfloor n/q\rfloor}^i(q)$ are (really) influential only in the $\OO(\sqrt{n})$ neighborhood of $n$ (Corollary~\ref{chernoffc} and Lemma~\ref{okolicentra}), where $p$-weights are only a bit ``downtrodden'' $q$-weights, it seems that the convergence of $\{a^p_n\}_{n\geq 0}$ could imply the convergence of $\{a^q_n\}_{n\geq 0}$. 

\begin{figure}[hbt]
\centering\includegraphics[width=8cm]{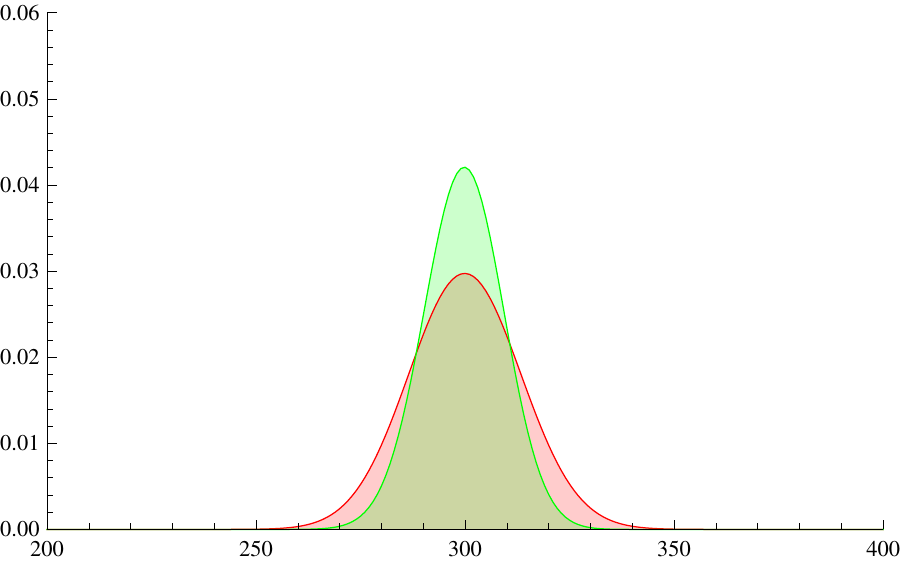}
\caption{The graphs show $B_{\lfloor n/p\rfloor}^i(p)$ and $B_{\lfloor n/q\rfloor}^i(q)$ with respect to $i$ in the neighborhood of $n$ for $n=300$, $p=0.4$ (red) and $q=0.7$ (green).}
	\label{slikaP}
\end{figure}

On the other hand, one could take $a_n=0$ for all except for some $n$ where there would be outliers of heights $\TT(\sqrt{n})$. Those outliers would be so far away that the weights $B_{n}^i(p)$ could ``notice'' two consecutive outliers, while the weights $B_{n}^i(q)$, which are slimmer, could not (on Fig.~\ref{slikaP}, two outliers could be at 280 and 320). Then $\{a_n^p\}_{n\geq 0}$ could converge because there would be a small difference between [when the weights $B_{n}^i(p)$ amplify one outlier] and [when they ``notice'' two outliers] (this two events seem to be the most opposite).  On the other hand, $\{a_n^q\}_{n\geq 0}$ would not converge.  From Chernoff bound (Corollary~\ref{chernoffc}) and from 
Lemma~\ref{okolicentra} it follows that the (horizontal) distance between outliers should be roughly $C\sqrt{n}$ for some $C$. What $C$ would be most appropriate?

\section{Comparing the limit of binomial means and the limit of arithmetic means}
	\label{compare}

This section contains the main result of this paper, which is formulated in the next theorem. The proof will be given in later.
\begin{theorem}
	\label{mainThm}
Let $\{a_n\}_{n\geq 0}$ be a sequence of non-negative 
real numbers such that $\{a^p_n\}_{n\geq 0}$ converges to $a\in\RR\cup\{\infty\}$ for some $0<p<1$. Then $\{a^*_n\}_{n\geq 0}$ converges to $a$.
\end{theorem}

An example of how this theorem can be used is given in Section~\ref{uporaba}. Here we give an example where $\{a^p_n\}_{n\geq 0}$ converges to $a\in\RR\cup\{\infty\}$ for all $0<p<1$, but $\{a^*_n\}_{n\geq 0}$ does not converge. 

\begin{proposition}
For the sequence $\{a_n\}_{n\geq 0}$ given by $a_n=(-1)^nn$, $\{a^p_n\}_{n\geq 0}$ converges to 0 for all $0<p<1$ and $\{a^*_n\}_{n\geq 0}$ does not converge.
\end{proposition}
\begin{proof}
For any $0<p<1$, it holds
\begin{align*}
a^p_n&=\sum_{i=0}^n(-1)^ii\binom{n}{i}p^i(1-p)^{n-i}\\
&=-n\frac{p}{1-p}\sum_{i=1}^n (-1)^{i-1}\binom{n-1}{i-1}p^{i-1}(1-p)^{n-(i-1)}\\
&=-n\frac{p}{1-p}(-p+(1-p))^{n-1}.
\end{align*}
Because $(-p+(1-p))=(1-2p)$ is strictly between $-1$ and 1, $\{a^p_n\}_{n\geq 0}$ converges to 0.

However, the reader can verify that $a^*_{2n+1}=-\frac{1}{2}$ and $a_{2n}^*=\frac{n}{2n+1}$, which implies that $\{a^*_n\}_{n\geq 0}$ does not converge. \qed
\end{proof}

Next, we show that we cannot interchange $\{a^p_n\}_{n\geq 0}$ and $\{a^*_n\}_{n\geq 0}$ in Theorem~\ref{mainThm}.
\begin{proposition}
There exists a sequence $\{a_n\}_{n\geq 0}$ of zeros and ones such that $\{a^*_n\}_{n\geq 0}$ converges to 0 and $\{a^p_n\}_{n\geq 0}$ diverges for all $0<p<1$.
\end{proposition}
\begin{proof}
Define
\begin{align*}
a_n =
\left\{
	\begin{array}{ll}
		1& \mbox{if there is some }k\in\NN \mbox{ such that } \left|n-2^{2k}\right|<2^kk\\
		0& \mbox{else.}
	\end{array}
\right.
\end{align*}
So $\{a_n\}_{n\geq 0}$ has islets of ones in the sea of zeros. The size of an islet at position $N$ is $\TT(\sqrt{N}\log(N))$ and the distance between two islets near position $N$ is $\TT(N)$. We leave to the reader to formally show that this implies the convergence of $a^*_n$ to zero.

Now let $0<p<1$. By Chernoff bound (Corollary~\ref{chernoffc}) we see that $B^i_n(p)$ is concentrated around $i=\lfloor np\rfloor$ and that for $|i-np|\geq\sqrt{n}\log(n)$, we have roughly nothing left. We leave to the reader to formally show that $\left\{a_{\lfloor2^{2k}/p\rfloor}^{p}\right\}_{k\geq 0}$ converges to 1 and that $\left\{a_{\lfloor2^{2k-1}/p\rfloor}^{p}\right\}_{k\geq 0}$ converges to 0, which implies that $\{a^p_n\}_{n\geq 0}$ diverges.\qed
\end{proof}

Now we go for the proof of Theorem~\ref{mainThm}. First, for a sequence $\{a_n\}_{n\geq 0}$ and $0<p<1$, we \textbf{define} $\{a_n^{p*}\}_{n\geq 0}$ as a sequence of arithmetic means of the sequence $\{a^p_n\}_{n\geq 0}$. We get
\begin{align*}
a_n^{p*}
&=\frac{1}{n+1}\sum_{j=0}^n a_j^p\\
&=\frac{1}{n+1}\sum_{j=0}^n\sum_{i=0}^j a_i\binom{j}{i}p^i(1-p)^{j-i}\\
&=\frac{1}{n+1}\sum_{i=0}^n a_i\sum_{j=i}^n\binom{j}{i}p^i(1-p)^{j-i}.
\end{align*}

It makes sense to \textbf{define} weights $w_n^i(p)=\sum_{j=i}^n\binom{j}{i}p^i(1-p)^{j-i},$ so that it holds
$$a_n^{p*}=\frac{1}{n+1}\sum_{i=0}^n w_n^i(p) a_i.$$

\begin{figure}[hbt]
\centering\includegraphics[width=8cm]{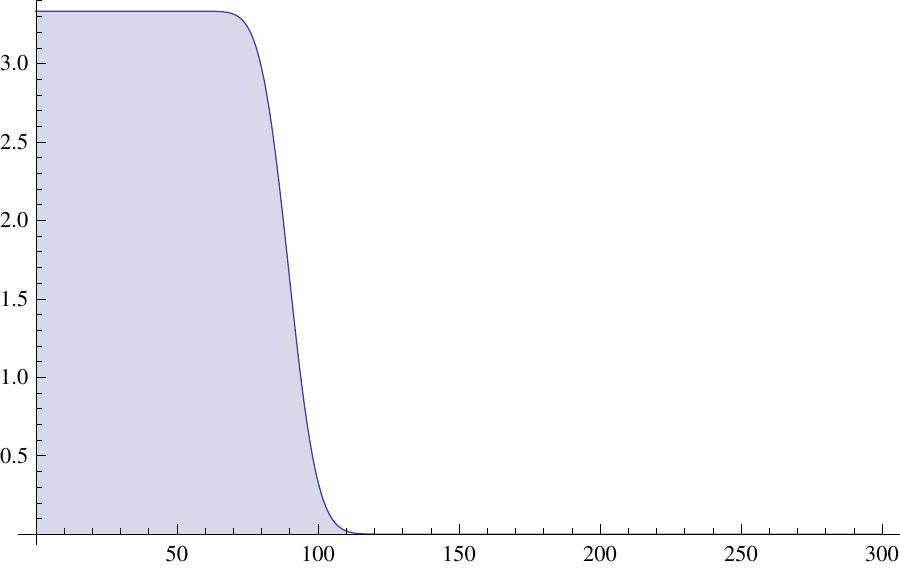}
\caption{The graph shows $w_{300}^i(0.3)$ with respect to $i$. We see a steep slope at $i=90$ reaching from height approximately $\frac{1}{0.3}$ to $0$.}
	\label{Wji}
\end{figure}

We can see on Fig.~\ref{Wji} that the weights $w_n^i(p)$ have a very specific shape. They are very close to $\frac{1}{p}$  for $i<np-\epsilon(n)$ and very close to 0 for $i>np+\epsilon(n)$, for some small $\epsilon(n)$. Such a shape can be well described using the next lemma (and its corollary), which gives another way to compute $w_n^i(p)$.
\begin{lemma}
	\label{utezi}
For $0<p<1$, $n\in\NN$ and $0\leq i\leq n$, it holds $$w_n^i(p)=\frac{1}{p}\left(1-\sum_{j=0}^i\binom{n+1}{j}p^j(1-p)^{n+1-j}\right).$$
\end{lemma}
\begin{proof} The idea is to use power series centered at $(1-p)$. For a function $f:\RR\rightarrow\RR$, we will write $f^{(i)}:\RR\rightarrow\RR$ for its $i$-th derivative.
\begin{align*}
w_n^i(p)
&=\sum_{j=i}^n\binom{j}{i}p^i(1-p)^{j-i}\\
&=\left.\frac{p^i}{i!}\left(\sum_{j=0}^n x^j\right)^{(i)}\right\vert_{x=1-p}\\
&=\left.\frac{p^i}{i!}\left(\frac{1-x^{n+1}}{1-x}\right)^{(i)}\right\vert_{x=1-p}\\
&=\left.\frac{p^{i-1}}{i!}\left(\frac{1-\big(x-(1-p)+(1-p)\big)^{n+1}}{1-\frac{1}{p}(x-(1-p))}\right)^{(i)}\right\vert_{x=1-p}\\
&=\frac{p^{i-1}}{i!}\left[\left(1-\sum_{k=0}^{n+1}\binom{n+1}{k}(x-(1-p))^k(1-p)^{n+1-k}\right)\right.\\
&\left.\left.\cdot\left(\sum_{k=0}^{\infty}(x-(1-p))^kp^{-k}\right)\right]^{(i)}\right\vert_{x=1-p}\\
&=\frac{p^{i-1}i!}{i!}\left(p^{-i}-\sum_{j=0}^i\binom{n+1}{j}(1-p)^{n+1-j}p^{j-i}\right)\\
&=\frac{1}{p}\left(1-\sum_{j=0}^i\binom{n+1}{j}p^j(1-p)^{n+1-j}\right).&\qed
\end{align*}
\end{proof}
\textbf{Define} the function $\epsilon:\NN\rightarrow\RR^+$ as 
\begin{align*}
\epsilon(n)=
\left\{
	\begin{array}{ll}
		\sqrt{n}\log(n)& \mbox{if }n\geq 2\\
		1& \mbox{else.}
	\end{array}
\right.
\end{align*}
 Now the following corollary holds.
\begin{corollary}
	\label{mejaUtezi}
For $0<p<1$, $n\in\NN$ and $0\leq i\leq n$, it holds
\begin{align*}
w_n^{\lfloor np-\epsilon(n)\rfloor}(p)&\geq\frac{1}{p}-n^{-\TT(\log(n))},\\
w_n^{\lfloor np+\epsilon(n)\rfloor}(p)&\leq n^{-\TT(\log(n))}.
\end{align*}
\end{corollary}
\begin{proof}
Use the Chernoff bound (Corollary~\ref{chernoffc}) on the expression for $w_n^i(p)$ from Lemma~\ref{utezi}.\qed
\end{proof}
For $0<p<1$ and for a sequence $\{a_n\}_{n\geq 0}$, \textbf{define} sequences $\{a_n^x(p)\}_{n\geq 0}$, $\{a_n^y(p)\}_{n\geq 0}$ and $\{a^z_n(p)\}_{n\geq 0}$ as
\begin{align*}
a^x_n(p)&=\sum_{i=0}^{\lfloor pn-\epsilon(n)\rfloor}w_n^i(p) a_i\\
a^y_n(p)&=\sum_{i=\lfloor pn-\epsilon(n)\rfloor+1}^{\lfloor pn+\epsilon(n)\rfloor-1}w_n^i(p) a_i\\
a^z_n(p)&=\sum_{i=\lfloor pn+\epsilon(n)\rfloor}^n w_n^i(p) a_i.
\end{align*}
Hence, we have
$$a_n^{p*}=\frac{1}{n+1}\Big(a^x_n(p)+a^y_n(p)+a^z_n(p)\Big).$$
From Corollary~\ref{mejaUtezi} we see that the weights in $a^x_n(p)$ are very close to $\frac{1}{p}$, which suggests that $\frac{1}{n+1}a^x_n(p)$ can be very close to $a^*_{\lfloor np \rfloor}$ (see Lemma~\ref{ax} below). From the same corollary we see that $\frac{1}{n+1}a^z_n(p)$ can be very close to 0 (see Lemma~\ref{az} below). And because we have a sum of only $\TT(\epsilon(n))$ elements in $a^y_n(p)$, $\frac{1}{n+1}a^y_n(p)$ could also be very close to 0 (see Lemma~\ref{ay} below).

We have just described the main idea for the proof of the main theorem, which we give next. It will use three lemmas just mentioned (one about $a^x_n(p)$, one about $a^y_n(p)$ and one about $a^z_n(p)$), that will be proven later.
\begin{proof}[Proof of Theorem~\ref{mainThm}.]
Suppose that $a_n\geq 0$ for all $n$ and suppose that $\{a_n^p\}_{n\geq 0}$ converges to  $a\in\RR\cup\{\infty\}$ for some $0<p<1$. 
We know that this implies the convergence of $\{a_n^{p*}\}_{n\geq 0}$ to $a$ (Theorem~\ref{enost}). 

First, we deal with the case $a=\infty$. We can use the fact that $w_n^i(p)\leq\frac{1}{p}$ for all $i$ (see Lemma~\ref{utezi}), which gives
\begin{align*}
a_n^{p*}
&=\frac{1}{n+1}\sum_{i=0}^n w_n^i(p) a_i\\
&\leq\frac{1}{n+1}\sum_{i=0}^n \frac {1}{p} a_i\\
&=\frac{a^*_n}{p}.
\end{align*}
Hence, $\{a_n^{*}\}_{n\geq 0}$ converges to $a=\infty$.

In the case $a<\infty$, we can use Lemma~\ref{ay} and Lemma~\ref{az} to see that $\left\{\frac{1}{n+1}a^y_n(p)\right\}_{n\geq 0}$ and $\left\{\frac{1}{n+1}a^z_n(p)\right\}_{n\geq 0}$ converge to 0. Hence, $\left\{\frac{1}{n+1}a^x_n(p)\right\}_{n\geq 0}$ converges to $a$. Lemma~\ref{ax} tells us that in this case $\{a_n^{*}\}_{n\geq 0}$ also converges to $a$. \qed
\end{proof}
Now we state and prove the remaining lemmas.
\begin{lemma}
	\label{ay}
Let $0<p<1$ and let $\{a_n\}_{n\geq 0}$ be a sequence of non-negative real numbers such that $\{a^p_n\}_{n\geq 0}$ converges to $a<\infty$. Then $\left\{\frac{1}{n+1}a^y_n(p)\right\}_{n\geq 0}$ converges to 0.
\end{lemma}
\begin{proof}
Fix $\tilde{\epsilon}>0$, define $\delta(n)=\lfloor\log^2(n)\rfloor$ and let $k:\NN\rightarrow\ZZ$ be such that $pn-\epsilon(n)\leq k(n)\leq pn+\epsilon(n)-\delta(n)$ holds for all $n$. We claim that 
$$\sum_{i=k(n)}^{k(n)+\delta(n)}a_i=\OO(\sqrt{n}),$$
where the constant behind the $\OO$ is independent of $k$.
To prove this, define $N=N(n)=\left\lfloor\frac{k(n)}{p}\right\rfloor$. It follows that $N=n\pm\TT(\epsilon(n))$. Note that, for large enough $n$,
$$\sum_{i=k(n)}^{k(n)+\delta(n)}a_i B_N^i(p)\leq\sum_{i=0}^{N}a_i B_N^i(p)< a+\tilde{\epsilon},$$
because $\{a^p_n\}_{n\geq 0}$ converges to $a$.
From Lemma~\ref{okolicentra} which bounds the coefficients $B_N^i(p)$ around $i=pN$ it follows that, for all $k(n)\leq i\leq k(n)+\delta(n)$,
$$B^{i}_N(p)\geq e^{-\oo(1)}B^{\lfloor Np\rfloor}_N(p).$$
Using $N=n\pm\TT(\epsilon(n))$ and the bound $$B^{\lfloor Np\rfloor}_N(p)=\frac{1}{\TT(\sqrt{N})}$$ from Lemma~\ref{osrednji}, we get
$$\sum_{i=k(n)}^{k(n)+\delta(n)}a_i <( a+\tilde{\epsilon}) e^{\oo(1)}\TT(\sqrt{n})=\OO(\sqrt{n}).$$

Next, we can see that 
$$\sum_{i=\lfloor pn-\epsilon(n)\rfloor+1}^{\lfloor pn+\epsilon(n)\rfloor-1}a_i=\OO\left(\frac{n}{\log n}\right).$$
Just partition the sum on the left-hand side into $\left\lceil\frac{2\epsilon(n)}{\delta(n)}\right\rceil$ sums of at most $\delta(n)$ elements. Then we have
$$\sum_{i=\lfloor pn-\epsilon(n)\rfloor+1}^{\lfloor pn+\epsilon(n)\rfloor-1}a_i=\OO\left(\frac{\epsilon(n)}{\delta(n)}\sqrt{n}\right)=\OO\left(\frac{n}{\log n}\right).$$

Now using $w_n^i(p)\leq\frac{1}{p}$ from Lemma~\ref{utezi}, we get
\begin{equation*}
\frac{1}{n+1}a^y_n(p)\leq \frac{1}{(n+1)p}\sum_{i=\lfloor pn-\epsilon(n)\rfloor+1}^{\lfloor pn+\epsilon(n)\rfloor-1} a_i= \frac{1}{(n+1)p}\OO\left(\frac{n}{\log n}\right),
\end{equation*}
which implies the convergence of $\left\{\frac{1}{n+1}a^y_n(p)\right\}_{n\geq 0}$ to 0.\qed
\end{proof}
\begin{lemma}
	\label{az}
Let $0<p<1$ and let $\{a_n\}_{n\geq 0}$ be a sequence of non-negative real numbers such that $\{a^p_n\}_{n\geq 0}$ converges to $a<\infty$. Then $\left\{\frac{1}{n+1}a^z_n(p)\right\}_{n\geq 0}$ converges to 0.
\end{lemma}
\begin{proof}
From Lemma~\ref{utezi} we see that the weights $w_n^i(p)$ decrease with $i$, so
\begin{equation*}
\frac{1}{n+1}a^z_n(p)\leq \frac{w_n^{\lfloor np+\epsilon(n)\rfloor}(p)}{(n+1)}\sum_{\lfloor pn+\epsilon(n)\rfloor}^n a_i.
\end{equation*}
 Corollary~\ref{mejaUtezi} gives us 
$w_n^{\lfloor np+\epsilon(n)\rfloor}(p)\leq n^{-\TT(\log(n))},$
while Lemma~\ref{koren} implies $a_i=\OO(\sqrt{i})$. Hence,  $\left\{\frac{1}{n+1}a^z_n(p)\right\}_{n\geq 0}$ converges to 0.\qed
\end{proof}
\begin{lemma}
	\label{ax}
Let $0<p<1$ and let $\{a_n\}_{n\geq 0}$ be a sequence of non-negative real numbers such that $\left\{\frac{1}{n+1}a^x_n(p)\right\}_{n\geq 0}$ converges to $a<\infty$. Then $\{a^*_n\}_{n\geq 0}$ converges to $a$.
\end{lemma}
\begin{proof}
Because the weights $w_n^i(p)$ are bounded from above by $\frac{1}{p}$ (Lemma~\ref{utezi}), we have
$$\frac{a^x_n(p)}{n+1}\cdot\frac{(n+1)p}{\lfloor pn-\epsilon(n)\rfloor+1}\leq a^*_{\lfloor pn-\epsilon(n)\rfloor},$$
where the left side converges to $a$.

Because the weights $w_n^i(p)$ decrease with $i$  (Lemma~\ref{utezi}) and because 
$w_n^{\lfloor np-\epsilon(n)\rfloor}(p)\geq\frac{1}{p}-n^{-\TT(\log(n))}$ (Corollary~\ref{mejaUtezi}), we have 
\begin{equation*}
a^*_{\lfloor pn-\epsilon(n)\rfloor}\leq\frac{a^x_n(p)}{n+1}\cdot\frac{n+1}{(\frac{1}{p}-n^{-\TT(\log(n))})\cdot(\lfloor pn-\epsilon(n)\rfloor+1)},
\end{equation*}
where the right side converges to $a$.
Hence, $a^*_{\lfloor pn-\epsilon(n)\rfloor}$ is sandwiched between two sequences that converge to $a$. It follows that $\{a^*_n\}_{n\geq 0}$ converges to $a$.\qed
\end{proof}
\section{Application of Theorem~\ref{mainThm}: a limit theorem for finite Markov chains}
	\label{applications}
For a stochastic matrix\footnote{A stochastic matrix is a (possibly infinite) square matrix that has non-negative real entries and for which all rows sum to 1. Each stochastic matrix represents transition probabilities of some discrete Markov chain. No prior knowledge of Markov chains is needed for this paper.
} 
$P$, define the sequence $\{P_n\}_{n\geq0}$ as $P_n=P^n$. As in the one-dimensional case, we define the sequence $\{P_n^*\}_{n\geq0}$ as $P_n^*=\frac{1}{n+1}\sum_{i=0}^nP_n$. We say that $\{P_n\}_{n\geq0}$ converges to $A$ if, for all possible pairs $(i,j)$, the sequence of $(i,j)$-th elements of $P_n$ converges to $(i,j)$-th element of $A$.
In this section, we will prove the following theorem.
\begin{theorem}
	\label{uporaba}
For any finite stochastic matrix $P$, the sequence $\{P^*_n\}_{n\geq0}$ converges to some stochastic matrix $A$, such that $AP=PA=A$.
\end{theorem}

This theorem is nothing new in the theory of Markov chains. Actually, it also holds for (countably)  infinite transition matrices $P$. Although we did not find it formulated this way in literature, it can be easily deduced from the known results. The hardest thing to show is the convergence of $\{P^*_n\}_{n\geq0}$~\cite[page 32]{MVerige}. After we have it, we can continue as in the proof of Theorem~\ref{uporaba} below.


We will give a short proof of Theorem~\ref{uporaba}, using only linear algebra and Theorem~\ref{mainThm}. First, we prove a result from linear algebra.
\begin{lemma}
	\label{souporaba}
Let $P$ be a finite stochastic matrix and let $\tilde{P}=\frac{1}{2}\Big(P+I\Big)$. Then
\begin{enumerate}[\hspace{0.4cm} a)]
\item for all eigenvalues $\lambda$ of $\tilde{P}$, it holds $|\lambda|\leq 1$,
	\label{souporaba1}
\item for all eigenvalues $\lambda$ of $\tilde{P}$ for which $|\lambda| =1$, it holds $\lambda = 1$,
	\label{souporaba2}
\item the algebraic and geometric multiplicity of eigenvalue 1 of $\tilde{P}$ are the same.
	\label{souporaba3}
\end{enumerate}
\end{lemma}
\begin{proof}
Since the product and convex combination of stochastic matrices is a stochastic matrix, $P^n$ and $\tilde{P}^n$ are stochastic matrices for each $n\in\NN$.
First, we will prove by contradiction that for all eigenvalues $\lambda$ for $P$, it holds $|\lambda|\leq 1$. Suppose that there is some eigenvalue $\lambda$ for $P$ such that $|\lambda|>1$. Let $w$ be the corresponding eigenvector and let its $i$-th component be non-zero. Then $|(P^nw)_i|=|\lambda^n|\cdot| w_i|$, where the right side converges to $\infty$ and the left side is bounded by $\max_j |w_j|$ (since $P^n$ is a stochastic matrix). This gives a contradiction. Hence, for all eigenvalues $\lambda$ for $P$, it holds $|\lambda|\leq 1$. Because $\tilde{P}$ is also stochastic, the same holds for $\tilde{P}$.

We see that we can get all eigenvalues of $\tilde{P}$ by adding 1 and dividing by 2 the eigenvalues of $P$. Because $P$ has all eigenvalues in the unit disc around 0, $\tilde{P}$ has all eigenvalues in a disc centered in $\frac{1}{2}$ of radius $\frac{1}{2}$. Hence, for all eigenvalues $\lambda$ of $\tilde{P}$, for which $|\lambda| =1$, it holds $\lambda = 1$.

For the last claim of the lemma, suppose that the algebraic and geometric multiplicity of eigenvalue 1 of $\tilde{P}$ are not the same. Then, by Jordan decomposition, there is an eigenvector $v$ for eigenvalue 1 and a vector $w$, such that $\tilde{P}w=v+w$. Then, for each $n\in\NN$, we have $\tilde{P}^nw=nv+w$. Because $v$ has at least one non-zero component and because all components of $\tilde{P}^nw$ are bounded by absolute value by $\max_j |w_j|$, we have come to contradiction. Hence, the algebraic and geometric multiplicity of eigenvalue 1 of $\tilde{P}$ are the same.\qed
\end{proof}
\begin{proof}[Proof of Theorem~\ref{uporaba}]
For a matrix $\tilde{P}=\frac{1}{2}\Big(P+I\Big)$, let $\tilde{P}=XJX^{-1}$ be its Jordan decomposition. From Lemma~\ref{souporaba}~\ref{souporaba1}) and~\ref{souporaba2}) it follows that the diagonal of $J$ consists only of ones and entries of absolute value strictly less than one. From Lemma~\ref{souporaba}~\ref{souporaba3}) it follows that the Jordan blocks for eigenvalue 1 are all $1\times 1$. It follows that $J^n$ converges to some matrix $J_0$ with only zero entries and some ones on the diagonal. Hence, $\tilde{P}^n$ converges to $A=XJ_0X^{-1}$. Since $\tilde{P}^n$ is a stochastic matrix for all $n$, the same is true for $A$. 
 Using $\tilde{P}_n=\tilde{P}^n$, we see that $\{\tilde{P}_n\}_{n\geq 0}$ is just a sequence of $0.5$-binomial means of the sequence $\{P_n\}_{n\geq 0}$, hence by Theorem~\ref{mainThm} $\{P^*_n\}_{n\geq0}$ also converges to $A$. Thus, we have
\begin{align*}
AP
&=\left(\lim_{n\rightarrow\infty} \frac{1}{n+1}\sum_{i=0}^nP^i\right)P\\
&=\lim_{n\rightarrow\infty} \frac{n+2}{n+1}\left(\frac{1}{n+2}\sum_{i=0}^{n+1}P^i-\frac{1}{n+2}I\right)\\
&=A.
\end{align*}
The same argument shows also $PA=A$.\qed
\end{proof}

\paragraph{Acknowledgements.} The author wishes to thank his research advisor Sergio Cabello, as well as Matjaž Konvalinka and Marko Petkovšek for valuable comments and suggestions.

\vspace{1cm}

\noindent All figures were done with Mathematica 8.0.1.0.
%
%
%

\bibliographystyle{abbrv}
\bibliography{literature}

\end{document}